\documentclass[a4paper]{amsart}

\usepackage{amsmath,amssymb,amsthm,stmaryrd}

\usepackage{color,graphicx}
\usepackage[dvipsnames]{xcolor}
\usepackage{accents}
\usepackage{enumerate}

\usepackage[textsize=footnotesize,color=yellow!40, bordercolor=white]{todonotes}

\newtheorem{theorem}{Theorem}[section]

\newcounter{cl}[theorem]
\newtheorem{claim}[cl]{Claim}

\newtheorem*{claim*}{Claim}
\newtheorem*{subclaim*}{Subclaim}

\theoremstyle{definition}

\newtheorem*{definition*}{Definition}
\newcounter{ca}[cl]

\newcounter{sca}[ca]

\newcounter{ssca}[sca]

\theoremstyle{remark}

\newtheorem*{remark*}{Remark}

\usepackage{tikz}
\usetikzlibrary{decorations.pathmorphing,shapes}

\usepackage[
    colorlinks=true, 
    citecolor=blue,
    linkcolor=blue,
    urlcolor=blue]{hyperref}
    
\newcommand{\ZFC}{\ensuremath{\operatorname{ZFC}} }
\newcommand{\ZF}{\ensuremath{\operatorname{ZF}} }

\newcommand{\AD}{\ensuremath{\operatorname{AD}} }

\newcommand{\cO}{\mathcal{O}}

\begin{document}
\title[Independence Phenomena in Mathematics]{Independence Phenomena in Mathematics:\\ a Set Theoretic Perspective on Current Obstacles and Scenarios for Solutions}

%\subjclass[2010]{03E45, 03E60, 03E25, 03E15} 

%\keywords{Infinite Game, Determinacy, Universally Baire, Inner Model
% Theory, Mouse}

\author{Sandra M\"uller} \address{Sandra M\"uller, Institut f\"ur Diskrete Mathematik und Geometrie, TU Wien, Wiedner Hauptstra{\ss}e 8-10/104, 1040 Wien, Austria.}
\email{sandra.mueller@tuwien.ac.at}
\thanks{This research was funded in whole or in part by the Austrian Science Fund (FWF) [10.55776/Y1498, 10.55776/I6087, 10.55776/V844]. For open access purposes, the author has applied a CC BY public copyright license to any author accepted manuscript version arising from this submission.}

\subjclass[2010]{} 


%\date{\today}

\begin{abstract}
  The standard axioms of set theory, the Zermelo-Fraenkel axioms (ZFC), do not suffice to answer all questions in mathematics. While this follows abstractly from Kurt Gödel's famous incompleteness theorems, we nowadays know numerous concrete examples for such questions. A large number of problems in set theory, for example, regularity properties such as Lebesgue measurability and the Baire property are not decided - for even rather simple (for example, projective) sets of reals - by ZFC. Even many problems outside of set theory have been showed to be unsolvable, meaning neither their truth nor their failure can be proven from ZFC. A major part of set theory is devoted to attacking this problem by studying various extensions of ZFC and their properties. We outline some of these extensions and explain current obstacles in understanding their impact on the set theoretical universe together with recent progress on these questions and future scenarios. This work is related to the overall goal to identify the ``right'' axioms for mathematics. 
\end{abstract}
\maketitle
\setcounter{tocdepth}{1}
% \tableofcontents

\section{Introduction}
Kurt Gödel predicted in the 1930's that the standard axioms of set theory, Zermelo-Fraenkel set theory ($\ensuremath{\operatorname{ZF}}$) with Choice ($\ensuremath{\operatorname{ZFC}}$), do not suffice to answer all questions in mathematics. More precisely, he showed in his first incompleteness theorem that there are statements that can neither be proven nor disproven from the $\ensuremath{\operatorname{ZF}}$ or $\ensuremath{\operatorname{ZFC}}$ axioms (or any sufficiently nice extension thereof).\footnote{Gödel showed his first incompleteness theorem in fact for an axiomatic system weaker than $\ensuremath{\operatorname{ZF}}$ but this will not be relevant in this article. We always assume here that there is a model of $\ZF$.} While Gödel's famous incompleteness theorems are abstract theoretical results, we nowadays know numerous concrete examples for such questions.
In addition to a large number of problems in set theory,
%In particular, regularity properties such as Lebesgue measurability and the Baire property are not decided -- for even rather simple (for example, projective) sets of reals -- by $\ensuremath{\operatorname{ZFC}}$. Even
even many problems outside of set theory are known to be unsolvable, meaning neither their truth nor their failure can be proven from $\ensuremath{\operatorname{ZFC}}$. We call such problems \emph{independent from $\ZFC$}. This includes the Whitehead Problem (group theory, \cite{Sh74}), the Borel Conjecture (measure theory, \cite{La76}), Kaplansky's Conjecture on Banach algebras (analysis, \cite{DW87}), and the Brown-Douglas-Fillmore Problem (operator algebras, \cite{Fa11}). A major part of set theory is devoted to attacking this problem by studying various extensions of $\ZF$ and their properties. One of the main goals of current research in set theory is to identify \emph{the ``right'' axioms for mathematics} that settle these problems. This, in part philosophical, problem is attacked with technical mathematical methods by analyzing various set theoretic universes satisfying different extensions of $\ZF$.
%\red{Was sind Kriterien für gute / die richtigen Axiome.}
% the $L$ space problem (\cite{MT2006}), a problem of von Neumann (ergodic theory, \cite{FRB}),

A widely useful and well-studied collection of such axioms is given by {\bf determinacy assumptions}: These are canonical extensions of $\ensuremath{\operatorname{ZF}}$ or $\ensuremath{\operatorname{ZFC}}$ that postulate the existence of winning strategies in natural two-player games \cite{GS53, MySt62}. For every set of infinite sequences of natural numbers\footnote{As common in set theory, we will tacitly identify infinite sequences of natural numbers with reals.} $A \subseteq \mathbb{N}^{\mathbb{N}}$, we consider an infinite two-player game where player $\mathrm{I}$ and player
$\mathrm{II}$ alternate playing natural numbers $n_0, n_1, \dots$, as follows: \vspace{-0.2cm}
\[ \begin{array}{c|ccccc} \mathrm{I} & n_0 & & n_2 & &\hdots \\ \hline
    \mathrm{II} & & n_1 & & n_3 & \hdots 
   \end{array} \]

Then player $\mathrm{I}$ wins the game if and only if the sequence $x = (n_0,n_1,\dots)$ of natural numbers produced during a run of the game is
an element of $A$; otherwise, player $\mathrm{II}$ wins. We call $A$ the \emph{payoff set} of this game. A set $A$ is called \emph{determined} if and only if
one of the two players has a winning strategy in the game with payoff set $A$, meaning that there is a method by
which they can win in the game described above, no matter what their opponent does. The \textbf{Axiom of Determinacy} ($\AD$) is the statement that all sets of reals are determined. Determinacy hypotheses are known to 
%imply regularity properties, and 
enhance sets of real numbers with a great deal of canonical structure. 
But while a pioneering result of Gale and Stewart shows that every open and every closed set in $\mathbb{N}^{\mathbb{N}}$ is determined in $\ZFC$, determinacy for more complex, for example, analytic, sets cannot be proven in $\ZFC$ alone. It should be mentioned at this point that the full Axiom of Determinacy as stated above contradicts the Axiom of Choice, so it can only be considered as an extension of $\ZF$, not of $\ZFC$.

Other natural and well-studied extensions of $\ZFC$ are given by the hierarchy of {\bf large cardinal axioms}. Measurable cardinals, for example, were introduced by Ulam \cite{Ul30} and further developed by Keisler and Scott \cite{Kei62, Sc61}. Determinacy assumptions, large cardinal axioms, and their consequences are widely used and have many fruitful implications in set theory and even in other areas of mathematics such as algebraic topology \cite{CSS05}, topology \cite{Ny80,Fl82,CMM1}, algebra \cite{EM02}, and operator algebras \cite{Fa11}. Many such applications, in particular, proofs of consistency strength lower bounds, exploit the interplay of large cardinals and determinacy axioms. Thus, understanding the connections between determinacy assumptions and the hierarchy of large cardinals is vital to answer questions left open by $\ZFC$ itself.

A third canonical line of axioms extending $\ZFC$ is given by {\bf forcing axioms}. They postulate that certain sets used to build canonical extensions of the set theoretic universe exist. The first forcing axioms that was isolated was Martin's Axiom introduced by Martin and inspired by Solovay's and Tennenbaum's proof of the independence of Souslin's Hypothesis \cite{MaSo70, SoTe71}. Other well-known examples with many applications are the Proper Forcing Axiom introduced by Baumgartner (see \cite{Mo10}) and the in some sense ``ultimate forcing axiom'' Martin's Maximum introduced by Foreman, Magidor, and Shelah \cite{FMS88}.

\section{An example: The Continuum Problem}

Arguably the most famous statement that is known to be independent from $\ZFC$ is Cantor's Continuum Problem. It was formulated by Cantor in 1878 and appeared as the first item on Hilbert's list of problems announced at the International Congress of Mathematicians in Paris in 1900. Informally, it can be phrased as the question how many real numbers there are. Or, a bit more formally, as the following question: Is there a set $A$ of size strictly between the size of the set of natural numbers $|\mathbb{N}|$ and the size of the set of real numbers $|\mathbb{R}|$? I.e., is there a set $A$ such that \[ |\mathbb{N}| < |A| < |\mathbb{R}|? \]
Gödel showed in the late 1930's that there are models of set theory in which there is no such set $A$. Many years later in 1963 Cohen was able to show that there are also models of set theory in which there is such a set $A$, thereby establishing the independence of the Continuum Problem from $\ZF$. He was awarded the Fields Medal in 1966 for this outstanding result.

We are interested in the question if and how canonical extensions of $\ZF$ decide the Continuum Problem. Interestingly, the three lines of extensions of $\ZF$ described in the introduction give different answers to the Continuum Problem.

\begin{description}
    \item[Determinacy] If the Axiom of Determinacy holds, then there is no set $A$ such that $|\mathbb{N}| < |A| < |\mathbb{R}|$.
    \item[Large Cardinals] Large cardinal axioms do not influence the Continuum Problem, i.e., it is still independent over $\ZFC$ together with the existence of large cardinals.
    \item[Forcing Axioms] Under the Proper Forcing Axioms there is a set $A$ such that  $|\mathbb{N}| < |A| < |\mathbb{R}|$. Even more, there is exactly one such intermediate size between  $|\mathbb{N}|$ and $|\mathbb{R}|$.
\end{description}

This might seem like it is the end of the dream to solve independence phenomena in mathematics by considering canonical hierarchies of axioms extending $\ZF$. In fact, this is only the beginning of a long and fascinating line of research. Instead of just considering the direct effect these extensions have on independent questions in mathematics, we can analyze how ``far away'' the extensions are from $\ZF$ by comparing their \textbf{consistency strength}.

\section{Connecting the hierarchies}

Before we can describe the connections between the hierarchies of determinacy axioms, large cardinal axioms, and forcing axioms, we outline what we mean by a ``hierarchy of axioms''. We illustrate the case for determinacy axioms as this is the easiest to explain without further set theoretic background. 

Recall that a set $A \subseteq \mathbb{N}^{\mathbb{N}}$ is called \emph{determined} iff one of the two players has a winning strategy in the infinite game with payoff set $A$. The Axiom of Determinacy ($\AD$) postulates that all sets of reals $A$ are determined but we can consider natural restrictions of $\AD$ by narrowing the collection of determined sets. In addition, we can consider strengthenings of $\AD$ by allowing more complicated games. This leads, for example, to the following (incomplete) list of axioms, ordered in increasing strength:

\begin{description}
    \item[Open and closed determinacy] Every open and every closed subset of  $\mathbb{N}^{\mathbb{N}}$ is determined.
    \item[Borel determinacy] Every Borel set $B \subseteq \mathbb{N}^{\mathbb{N}}$ is determined.
    \item[Analytic determinacy] Every analytic set $A \subseteq \mathbb{N}^{\mathbb{N}}$ is determined. Here a set $A$ is called \emph{analytic} iff it is the projection of a Borel subset of $\mathbb{N}^{\mathbb{N}} \times \mathbb{N}^{\mathbb{N}}$ to its first coordinate. In general, not every analytic set is Borel.
    \item[Projective determinacy] Every projective set $A \subseteq \mathbb{N}^{\mathbb{N}}$ is determined. Here a set $A$ is called \emph{projective} if it can be obtained from a Borel subset of $(\mathbb{N}^{\mathbb{N}})^n$ for some $n \in \mathbb{N}$ by an iterative alternation of taking complements and projections. By varying $n \in \mathbb{N}$ there is in fact a hierarchy of projective sets beyond the analytic sets.
    \item[Determinacy in $L(\mathbb{R})$] Every set of reals in $L(\mathbb{R})$ is determined. Here $L(\mathbb{R})$ denotes the smallest model of $\ZF$ that contains all reals and all ordinals.
    \item[Axiom of Determinacy ($\AD$)] Every set $A \subseteq \mathbb{N}^{\mathbb{N}}$ is determined.
    \item[Axiom of Determinacy for games on reals ($\AD_{\mathbb{R}}$)] Consider the infinite game where the players are allowed to play real numbers (instead of just natural numbers). Then for every payoff set $A \subseteq \mathbb{R}^{\mathbb{N}}$, one of the players has a winning strategy.
    \item[Axiom of Determinacy for games of uncountable length] Consider the infinite game where the players play uncountably many natural numbers, i.e., the moves in the game are enumerated by ordinals $\alpha < \omega_1$ (instead of just natural numbers). Then for every payoff set $A \subseteq \mathbb{N}^{\omega_1}$, one of the players has a winning strategy.
\end{description}

The first levels of this hierarchy are still provable in $\ZFC$. It was shown by Gale and Stewart in $\ZF$ that every open and every closed subset of  $\mathbb{N}^{\mathbb{N}}$ is determined \cite{GS53}. Here we refer to open and closed in the product topology on $\mathbb{N}^{\mathbb{N}}$ with the discrete topology on $\mathbb{N}$. I.e., for $s \in \mathbb{N}^{n}$ for some natural number $n$, let \[\cO(s) = \{x \in \mathbb{N}^{\mathbb{N}} \mid s \subset x\},\] where $s \subset x$ for $s$ and $x$ as above denotes that $s$ is an initial segment of $x$. Then we say that a set $A \subseteq \mathbb{N}^{\mathbb{N}}$ is \emph{open} iff it is a union of such basic open sets $\cO(s)$ for some finite sequence of natural numbers $s$. Moreover, $A$ is \emph{closed} iff $\mathbb{N}^{\mathbb{N}} \setminus A$ is open.

\begin{theorem}[Gale, Stewart, 1953, \cite{GS53}]
   If $A \subseteq \mathbb{N}^{\mathbb{N}}$ is open or closed, then the game with payoff set $A$ is determined.
 \end{theorem}

 \begin{proof}
   The basic idea of the proof is that membership in an open set is secured at a finite stage and games where the winner is decided at a finite stage are determined.

   In what follows we write \[ B / s = \{ x \in \mathbb{N}^{\mathbb{N}} \mid s^\frown x \in B\}\] for all $B \subseteq \mathbb{N}^{\mathbb{N}}$ and $s \in \mathbb{N}^{n}$ for some $n \in \mathbb{N}$. We will use the same notation with the obvious meaning for $B \subseteq \mathbb{N}^{n}$ for $n \in \mathbb{N}$ as well.

   \begin{claim} \label{cl:opendet} Let $s \in \mathbb{N}^{2n}$.
     If I does not have a winning strategy in the game with payoff set $B/s$, then for any $i \in \mathbb{N}$, there is a $j \in \mathbb{N}$ such that I does not have a winning strategy in the game with payoff set $B/s^\frown(i,j)$.
   \end{claim}
   \begin{proof}
     Suppose this is not the case. Let $i \in \mathbb{N}$ be a counterexample, i.e., for any $j \in \mathbb{N}$, I has a winning strategy $\sigma_j$ in the game with payoff set $B/s^\frown(i,j)$.\footnote{Choosing these $\sigma_j$ for each $j$ can be done in a canonical way to avoid using the Axiom of Choice.} Then the following is a winning strategy for I in the game with payoff set $B/s$: Start by playing $i$. Then if II responds with some $j \in \mathbb{N}$, play according to $\sigma_j$ for the rest of the game.
   \end{proof}

   Now say $A \subseteq \mathbb{N}^{\mathbb{N}}$ is open and suppose that I does not have a winning strategy in the game with payoff set $A$. We can apply Claim \ref{cl:opendet} recursively to obtain a strategy $\tau$ for II such that for any partial play $s \in \bigcup_{n\in\mathbb{N}} \mathbb{N}^{2n}$ according to $\tau$, I does not have a winning strategy in the game with payoff set $A/s$.

   Suppose $\tau$ is not a winning strategy for II, i.e., there is a play $x$ according to $\tau$ such that $x \in A$. As $A$ is open, there is some $2n\in \mathbb{N}$ such that $\cO(x\upharpoonright 2n) \subseteq A$. But then any strategy for I in the game with payoff set $A/ (x\upharpoonright 2n)$ would be a winning strategy, contradicting the fact that $x$ is played according to $\tau$.
   Therefore, $\tau$ is a winning strategy for II in the game with payoff set $A$.

   If $A \subseteq \mathbb{N}^{\mathbb{N}}$ is closed, we can use an analogous argument with the roles of I and II interchanged.
 \end{proof}

 Martin showed in 1975 that also Borel determinacy can be proven in $\ZFC$ \cite{Ma75}. But proving stronger forms of determinacy requires extensions of $\ZFC$. For example, Martin showed in 1970 that, assuming the existence of a measurable cardinal, every analytic set of reals is determined \cite{Ma70}. Harrington showed a few years later that the measurable cardinal in the hypothesis of Martin's theorem is in some sense necessary \cite{Ha78}. In fact, the following equivalence is the first level of a surprising as well as deep connection between the hierarchies of determinacy axioms and large cardinals.

 \begin{theorem}[Harrington, Martin, \cite{Ha78, Ma70}]\label{thm:HaMa}
  The following are equivalent:
  \begin{enumerate}
  \item All analytic sets are determined, and
  \item for all $x \in \mathbb{N}^{\mathbb{N}}$, $x^\#$ exists.
  \end{enumerate}
\end{theorem}

Defining $x^\#$ (``$x$ sharp'') goes beyond the scope of this survey but we should mention that the fact that $x^\#$ exists for all $x \in \mathbb{N}^{\mathbb{N}}$ is a consequence of the existence of a measurable cardinal.

This deep connection between the hierarchies of determinacy axioms and large cardinals persists throughout the projective hierarchy \cite{Ne95, Ne02, MSW} and even beyond. For example, it is known how to construct models of $\AD$ or $\AD_{\mathbb{R}}$ from large cardinals. But this cannot be strengthened arbitrarily: It was shown by Mycielski in the 1960's already that the Axiom of Determinacy for games of uncountable length is inconsistent with $\ZF$.

Connections between the hierarchies of large cardinals and forcing axioms appear at a level in the large cardinal hierarchy much higher than the ones corresponding to $\AD$ or $\AD_{\mathbb{R}}$. For example, the existence of a supercompact cardinal implies that there is a model of Martin's Maximum or the Proper Forcing Axiom. In the other direction, it is known that the Proper Forcing Axiom implies the existence of models with certain large cardinals well below supercompact cardinals. The exact correspondence between large cardinals and forcing axioms is a wide open problem and has already motivated several ground breaking results in set theory in the past decades.

\section{Strong axioms of determinacy}

Most results obtaining so-called \emph{large cardinal lower bounds} for forcing axioms, for example, constructing a model with certain large cardinal from the Proper Forcing Axiom, exploit the deep connection between the large cardinal hierarchy and the determinacy hierarchy. Instead of directly building models with large cardinals, one can, for example, build a model of analytic determinacy, and then use Theorem \ref{thm:HaMa} to obtain large cardinal strength. One reason why this method currently does not succeed in reaching an exact correspondence between the Proper Forcing Axiom and some large cardinal axiom is that the there are no known determinacy axioms that are sufficiently strong to correspond to the levels in the large cardinal hierarchy in question. This leads to the following question:
\[ \text{What are possible ways to strengthen the Axiom of Determinacy?} \]
There are two scenarios how such a strengthening could look like.

\begin{enumerate}
    \item Consider restricted versions of the Axiom of Determinacy for games of uncountable length that are still consistent with $\ZF$. For example, games where the length of the game is only decided during the game but is guaranteed to be countable or games of uncountable length with definable payoff sets.
    \item Keep playing the same games of length $\mathbb{N}$ where the players play natural numbers or reals but at the same time impose additional structural properties on the models. 
\end{enumerate}

There are several promising results for both scenarios (see, for example, \cite{AgMu19, Ne04, Mu21, MS23}) and providing strengthenings of the Axiom of Determinacy that correspond to high levels in the large cardinal hierarchy is among the most important questions at the frontier of modern set theoretic research.

\bibliographystyle{plain}
\bibliography{References}

\end{document}